\documentclass[reqno,10pt]{amsart}
\usepackage{amsmath,amssymb}
\usepackage[mathscr]{euscript}

\theoremstyle{plain}
\newtheorem{theorem}{Theorem}[section]

\newtheorem{lemma}[theorem]{Lemma}

\numberwithin{equation}{section}
\numberwithin{theorem}{section}
\newcommand{\mc}[1]{{\mathcal #1}}
\newcommand{\mb}[1]{{\mathbf #1}}
\newcommand{\ms}[1]{{\mathscr #1}}

\DeclareMathOperator{\gap}{gap}
\newcommand{\ind}{\mathbf{1}}
\newcommand{\upbar}[1]{\,\overline{\! #1}}
\newcommand{\arginf}{\mathrm{arg \,inf}\:}
\renewcommand{\epsilon}{\varepsilon}
\renewcommand{\tilde}{\widetilde}

\title[Dynamical behavior of the ABC model]
{On the dynamical behavior of the ABC model} 

\author [L.\ Bertini]{Lorenzo Bertini}
\address{Lorenzo Bertini \hfill\break \indent
  Dipartimento di Matematica, Universit\`a di Roma `La Sapienza',
  \hfill\break \indent
  P.le Aldo Moro 5, 00185 Roma, Italy}
\email{bertini@mat.uniroma1.it}

\author [N.\ Cancrini]{Nicoletta Cancrini}
\address{Nicoletta Cancrini\hfill\break \indent
  Dipartimento di Matematica Pura e Applicata, Universit\`a dell'Aquila,
  \hfill\break \indent
  67100 Coppito, L'Aquila, Italy}
\email{nicoletta.cancrini@univaq.it}

\author [G.\ Posta]{Gustavo Posta}
\address{Gustavo Posta\hfill\break \indent
  Dipartimento di Matematica, Politecnico di Milano
  \hfill\break \indent
  P.za Leonardo da Vinci 32, 20133 Milano, Italy }
\email{gustavo.posta@polimi.it}

\begin{document}

\noindent


\begin{abstract}
We consider the ABC dynamics, with equal density of the three species,
on the discrete ring with $N$ sites.
In this case, the process is reversible with respect to a Gibbs measure
with a mean field interaction that undergoes a second order phase
transition. 
We analyze the relaxation time of the dynamics and show that at high
temperature it grows at most as $N^2$ while it grows at least as
$N^3$ at low temperature.
\end{abstract}

\maketitle
\thispagestyle{empty}

\section{Introduction}

The \emph{ABC model}, introduced by Evans et al.\ \cite{5,6}, is a
one-dimensional stochastic conservative dynamics with local jump rates
whose invariant measure undergoes a phase transition.  It is a system
consisting of three species of particles, traditionally labeled $A$,
$B$, $C$, on a discrete ring with $N$ sites. The system evolves by
nearest neighbor particles exchanges with the following (asymmetric)
rates: $AB \to BA$, $BC \to CB$, $CA \to AC$ with rate $q\in (0,1]$
and $BA \to AB$, $CB \to BC$, $AC \to CA$ with rate $1/q$.  In
particular, the total numbers of particles $N_\alpha$,
$\alpha\in\{A,B,C\}$, of each species are conserved and satisfy
$N_A+N_B+N_C=N$. Observe that the case $q=1$ corresponds to a three
state version of the symmetric simple exclusion process.  When
$q\in(0,1)$, Evans et al.\ \cite{5,6} argued that in the thermodynamic
limit $N \to \infty$ with $N_\alpha /N \to r_\alpha$ the system
segregates into pure $A$, $B$, and $C$ regions, with translationally
invariant distribution of the phase boundaries.  In the equal densities
case $N_A=N_B=N_C=N/3$ the dynamics is reversible and its invariant
measure can be explicitly computed.
As shown in \cite{FF1,FF2}, the ABC model can be reformulated 
terms of a dynamic of random walks on the triangular lattice.

As discussed by Clincy et al.\ \cite{8}, the natural scaling to
investigate the asymptotic behavior of the ABC model is the
\emph{weakly asymmetric} regime $q=\exp\big\{-\frac{\beta}{2N}\big\}$,
where the parameter $\beta\in [0,+\infty)$ plays the role of an
inverse temperature.  With this choice the reversible measure of the
equal densities case $r_A=r_B=r_C=1/3$ becomes a canonical Gibbs
measure, that we denote by $\nu_N^\beta$, with a mean field
Hamiltonian.  
The measure $\nu_N^\beta$ undergoes a second order phase transition at
$\beta_{\mathrm{c}}= 2\pi\sqrt{3}\approx 10.88$. 
This phase transition has been further analyzed in \cite{Ayetal} and
it is described in terms of the free energy functional $\mc F_\beta$
associated to $\nu_N^\beta$. 
The functional $\mc F_\beta$ is (apart an additive constant) the large
deviations rate function for $\nu_N^\beta$ in the scaling limit in
which the discrete ring with $N$ sites is embedded in the
one-dimensional torus and the particles configuration is described in
terms of the corresponding densities profiles $(\rho_A,\rho_B,\rho_C)$.  
In this limit, $\mc F_\beta(\rho_A,\rho_B,\rho_C)$ thus gives the
asymptotic probability of observing the density profile
$(\rho_A,\rho_B,\rho_C)$. 
In particular, the minimizer of $\mc F_\beta$ describes the typical
behavior of the system as $N\to\infty$.  The phase transition of
$\nu_N^\beta$ corresponds to the following behavior of the free energy
functional \cite{Ayetal,8,FF1}.  For $\beta\in[0,\beta_\mathrm{c}]$
the minimum of $\mc F_\beta$ is uniquely achieved at the homogeneous
profile $(1/3,1/3,1/3)$.  For $\beta>\beta_{\mathrm{c}}$ the
functional $\mc F_\beta$ has a continuum of minimizers, parameterized
by the translations, which describes the phase segregation.  As shown
in \cite{bdlw}, this phase transition can also be detected via the
two-point correlation functions of $\nu_N^\beta$ which become singular
when the system approaches the transition.  As shown in
\cite{Ayetal,FF1}, for $\beta>2\beta_\mathrm{c}$ the functional $\mc
F_\beta$ has other critical points besides the homogeneous profile and
the one-parameter family of minimizers.

For unequal densities the invariant measure of the ABC dynamics on a
ring is not reversible, that is the stationary state is no longer an
equilibrium state, and cannot be computed explicitly. As discussed in
\cite{bdlw,8}, a stability analysis of the homogeneous density profile
shows that for $\beta>2\pi\big[1-2(r_A^2+r_B^2+r_C^2)\big]^{-1/2}$ it
becomes unstable.  As stated there, one however expects that the phase
transition, at least for particular values of the parameters $r_A$,
$r_B$, $r_C$, becomes of the first order.  Again in \cite{bdlw,8}, the
asymptotic of the two-point correlation functions is computed in the
homogeneous phase and the large deviation rate function $\mc F_\beta$
has been calculated up to order $\beta^2$.  When the ABC dynamics is
considered on an open interval with reflecting endpoints, the
corresponding invariant measure is reversible for all values of the
densities \cite{Ayetal}.  In particular, it has the same Gibbs form as
the one in the ring for the equal density case.

The main purpose of the present paper is the discussion of the phase
transition of the ABC model on a ring with $N$ sites from a dynamical
viewpoint. More precisely, we focus on the asymptotic behavior, as
$N$ diverges, of the \emph{relaxation time} $\tau_N^\beta$ which
measures the time the dynamics needs to reach the stationary
probability.  Our analysis is restricted to the equal density case
$r_A=r_B=r_C=1/3$ in which the invariant measure $\nu_N^\beta$ is
explicitly known and reversible. As usual, the relaxation time
$\tau_N^\beta$ is defined as the inverse of the \emph{spectral gap} of
the generator $L_N^\beta$ of the underlying Markov process. Observe
that, in view of the reversibility, $L_N^\beta$ is a selfadjoint
operator on $L^2(d\nu_N^\beta)$. Our main result implies that the asymptotic
behavior of the relaxation time $\tau_N^\beta$ reflects the phase
transition of the corresponding stationary measure $\nu_N^\beta$. We
indeed show that for $\beta$ small enough $\tau_N^\beta$ is at most of
order $N^2$ while for $\beta>\beta_\mathrm{c}$ it is at least of order
$N^3$.

The diffusive behavior $\tau_N^\beta\sim N^2$ is characteristic of
conservative dynamics in the high temperature regime, the typical
example being the Kawasaki dynamics for the Ising model. Indeed, this
has been proven by different techniques in several contexts, see e.g.\
\cite{bcdp,CM,LY}.  We here follow the approach introduced in
\cite{bcdp} which is based upon a perturbative argument in $\beta$ and
can be directly applied to the case of mean field interactions.  
On the other hand, the behavior $\tau_N^\beta\sim N^3$ in the
supercritical regime is characteristic of the system under
consideration; we briefly discuss the heuristic picture.
As discussed in \cite{8} and then also in \cite{bdlw}, at time
$O(N^2)$ the densities 
profiles of the three species $(\rho_A,\rho_B,\rho_C)$ evolve
according to the hydrodynamic equations
\begin{equation}
  \label{hy}
  \begin{split}
    &\partial_t \rho_A + \beta\,\nabla\big[ \rho_A(\rho_C-\rho_B)\big]
    = \Delta \rho_A
    \\
    &\partial_t \rho_B + \beta \,\nabla \big[ \rho_B
    (\rho_A-\rho_C)\big] = \Delta \rho_B
    \\
    &\partial_t \rho_C + \beta \, \nabla \big[ \rho_C
    (\rho_B-\rho_A)\big] = \Delta \rho_C
  \end{split}
\end{equation}
where $\nabla$ and $\Delta$ denote the gradient and Laplacian on the
continuous torus, respectively.
As follows from microscopic reversibility, the evolution \eqref{hy}
can be obtained as a suitable gradient flow of the free energy $\mc
F_\beta$. In particular, while the homogeneous profile $(1/3,1/3,1/3)$
is the unique, globally attractive, stationary solution to \eqref{hy}
for $\beta<\beta_\mathrm{c}$, the (one parameter family of) minimizers of
$\mc F_\beta$ are stationary solutions to \eqref{hy} when
$\beta>\beta_\mathrm{c}$.  According to the \emph{fluctuating
  hydrodynamic theory}, we argue that, for large but finite $N$, the
hydrodynamic equation \eqref{hy} gives an accurate description of the
system provided one adds in \eqref{hy} a suitable noise term
$O\big(1/\sqrt{N}\big)$.  At time $O(N^2)$ the ABC dynamics then behaves
as a Brownian motion on the set of minimizers of $\mc F_\beta$ with
diffusion coefficient proportional to $1/N$, see  \cite{bd}.
The time to thermalize is thus $O(N^3)$.

The above scenario accounts for the correct asymptotics of the relaxation
time as long as there are no other local minima of $\mc F_\beta$.
As proven in \cite{Ayetal}, this is certainly the case for
$\beta\in(\beta_\mathrm{c}, 2\beta_\mathrm{c}]$.  
On the other hand, for $\beta>2\beta_\mathrm{c}$ other critical points
of $\mc F_\beta$ do appear but it is not known if they correspond to
local minima (this is indeed an open question in \cite{Ayetal}).
Were they local minima the ABC model would exhibit a metastable
behavior, i.e.\ starting in a neighborhood of such local minima the process
would spend a time exponential in $N$ in that neighborhood before reaching
the global minimizer.  Numerical evidences \cite{cf} suggest such
metastable behavior which would imply $\tau_N^\beta \sim \exp\{ c N\}$
for $\beta>2\beta_\mathrm{c}$.

\section{Notation and results}
\label{s:nr}

\subsection*{The ABC process}
Given a positive integer $N$, we let $\mb Z_N=\{0,\cdots,N-1\}$ be the
ring of the integers modulo $N$.  The configuration space with $N$
sites is $\tilde\Omega_N:=\{A,B,C\}^{\mb Z_N}$, elements of
$\tilde\Omega_N$ are denoted by $\zeta$, for $x\in\mb Z_N$ the species
of the particle at the site $x$ is thus $\zeta(x)\in \{A,B,C\}$.  We
also let $\eta_{\alpha}\colon \tilde\Omega_N \to \{0,1\}^{\mb Z_N}$,
$\alpha\in \{A,B,C\}$, be the $\alpha$ occupation numbers namely,
$[\eta_{\alpha}(\zeta)] (x) := \ind_{\{\alpha\}}(\zeta(x))$ in which $\ind_E$
stands for the indicator function of the set $E$.  Note that for each
$x\in\mb Z_N$ we have $\eta_{A}(x)+\eta_{B}(x)+\eta_{C}(x) =1$.  Whereas
$\eta=(\eta_A,\eta_B,\eta_C)$ is a function of the configuration
$\zeta$ we shall omit to write explicitly the dependence on
$\zeta$.

Given $x,y\in \mb Z_N$ and $\zeta\in \tilde \Omega_N$ we denote by
$\zeta^{x,y}$ the configuration obtained from $\zeta$ by exchanging
the particles at the sites $x$ and $y$, i.e.\ 
\begin{equation}
\label{ex}
  \big(\zeta^{x,y}\big) \, (z) :=
  \begin{cases}
    \zeta (y)  & \text{ if $\;z=x$}, \\
    \zeta (x)  & \text{ if $\;z = y$}, \\
    \zeta (z)  & \text{ otherwise}.
  \end{cases}
\end{equation}
The ABC process is the continuous time Markov chain on the state space
$\tilde\Omega_N$ whose generator $L_N=L_{N}^\beta$ acts on functions
$f\colon\tilde\Omega_N \to \mb R$ as
\begin{equation}
\label{L=}
  L_N^\beta f (\zeta) = \sum_{x\in\mb Z_N} c^\beta_x(\zeta) 
  \big[ f(\zeta^{x,x+1}) - f(\zeta)\big]
\end{equation}
where, for $\beta\ge 0$ the jump rates $c^\beta_x=c_{x}^{\beta,N}$ are given by 
\begin{equation}
  \label{rates}
  c_{x}^{\beta,N}(\zeta) :=
  \begin{cases}
    \exp\{ -\frac \beta{2N}\big\} & \text{ if $\; (\zeta(x),\zeta(x+1)) \in 
      \{ (A,C), (C,B), (B,A)\}$} \\ 
     \exp\{ \frac \beta{2N}\big\} &\text{ otherwise.}
  \end{cases}
\end{equation}

As follows from \eqref{L=}, the total number of particles of each
species is conserved.  Therefore, given three positive integers
$N_\alpha$, $\alpha\in\{A,B,C\}$ such that $N_A+N_B+N_C =N$, we have a
well defined process on the linear manifold $\sum_{x\in\mb Z_N}
\eta_\alpha(x)=N_\alpha$, $\alpha\in\{A,B,C\}$.  The ABC dynamics is
irreducible when restricted to such manifold; hence the process is
ergodic and admits a unique invariant measure.  In the case $\beta=0$
this measure is the uniform probability.  On the other hand, when
$\beta>0$ the explicit expression of the invariant measure is in
general not known.  However, as we next discuss, in the case
$N_A=N_B=N_C$ the ABC process satisfies the detailed balance condition
with respect to a mean field Gibbs measure \cite{5,6}.

\subsection*{Invariant measure in the equal densities case}
We assume that $N$ is a multiple of $3$ and we
restrict to the case in which $N_A=N_B=N_C$. We shall then consider
the ABC process on 
\begin{equation}
  \label{Om}
  \Omega_N:= \Big\{ \zeta\in \tilde \Omega_N \, : \: 
  \sum_{x\in\mb Z_N} \eta_A(x)=\sum_{x\in\mb Z_N}\eta_B(x)
  =\sum_{x\in\mb Z_N}\eta_C(x)= \frac N3
  \Big\}.  
\end{equation}
The Hamiltonian $H_N\colon\Omega_N\to \mb R$ is defined by
\begin{equation}
  \label{ham}
  H_N(\zeta) := \frac{1}{N^2} 
  \sum_{0\le x<y\le N-1}\big[ 
  \eta_A(x) \eta_C(y) +\eta_B(x) \eta_A(y) +\eta_C(x) \eta_B(y) 
  \big].
\end{equation}
In view of the equal densities constraint, an elementary computation
shows that the right hand side above does not depend on the choice of
the origin. Equivalently, $H_N$ is a translation invariant function on
$\Omega_N$.
Given $\beta\ge 0$, we denote by $\nu_N^\beta$ the
probability measure on $\Omega_N$ defined by
\begin{equation}
  \label{gm}
  \nu_N^\beta(\zeta) := \frac{1}{Z_N^\beta}\exp\big\{-\beta N H_N(\zeta)\big\}
\end{equation}
where $Z_N^\beta$, the partition function, is the proper
normalization constant. 
In the sequel, given a function $f$ on $\Omega_N$ we denote
respectively by $\nu_N^\beta(f)$ and $\nu_N^\beta(f,f)$ the
expectation and variance of $f$ with respect to $\nu_N^\beta$.

As observed in \cite{5,6}, the ABC process is reversible with respect
to $\nu_N^\beta$. In other worlds, the generator $L_N^\beta$ in
\eqref{L=} is a selfadjoint operator on $L^2(\Omega_N,d\nu_N^\beta)$
and in particular $\nu_N^\beta$ is the invariant measure.

\subsection*{Asymptotics of the spectral gap}

The spectrum of $L_N^\beta$ in \eqref{L=}, 
considered as a selfadjoint operator on $L^2(\Omega_N,\nu_N^\beta)$,
is a finite subset of the negative real axes and, in view of the
ergodicity of the associated process, zero is a simple eigenvalue of
$L_N^\beta$. 
The \emph{spectral gap} of $L_N^\beta$, denoted by $\gap(L_N^\beta)$,
is the absolute value of the second largest eigenvalue.  
The spectral gap can be characterized in variational terms as
follows: $\gap(L_N^\beta)$ is largest constant $\lambda\ge 0$ 
such that the Poincar\'e inequality  
\begin{equation}
  \label{rr}
  \lambda \, \nu_N^\beta(f,f) \le \nu_N^\beta \big( f (-L_N^\beta) f\big)
\end{equation}
holds for any $f\in L^2(\Omega_N,d\nu_N^\beta)$. 
The spectral gap controls the speed of convergence to equilibrium of
the associated process in the following sense. For each $f\in
L^2(\Omega_N,d\nu_N^\beta)$,
\begin{equation*}
  \nu_N^\beta \big(e^{t L_N^\beta}f \,, \,e^{t L_N^\beta} f \big)
  \le e^{ - 2\, \gap(L_N^\beta) \, t}  \: \nu_N^\beta (f,f).
\end{equation*}

Our main result concerns the asymptotic behavior of $\gap(L_N^\beta)$
as $N$ diverges. In particular we show this behavior differs in the
subcritical and supercritical regimes.

\bigskip 

\begin{theorem} $~$
  \label{t:mt}
  \begin{itemize}
  \item[(i)]
    There exist constants $\beta_0,C_0>0$ such that for any
    $\beta\in [0,\beta_0]$ and any $N$
    \begin{equation}
      \label{i}
      \gap(L_{N}^\beta)  \ge  C_0\, \frac {1}{N^2}\cdot
    \end{equation}
  \item[(ii)] 
    Let $\beta_\mathrm{c}:=2\pi \sqrt{3}$. For each 
    $\beta> \beta_\mathrm{c}$ there exists a constant $C(\beta)>0$
    such that for any $N$  
    \begin{equation}
      \label{ii}
      \gap(L_N^\beta)  \le  {C (\beta) } \, \frac 1 {N^3} \cdot
    \end{equation}
  \end{itemize}
\end{theorem}

The above statement raises few natural issues.   
As discussed in the Introduction, the $1/N^2$ asymptotic of the
spectral gap is a common feature of conservative stochastic dynamics
in the high temperature regime. Indeed, as proven in \cite{Q} for the
simple exclusion process and in \cite{CM,LY} for high temperature
Kawasaki dynamics, the spectral gap admits an upper bound that matches
\eqref{i}. 
We expect that the diffusive behavior $ \gap(L_{N}^\beta) =O(1/N^2)$
holds for any $\beta\in [0,\beta_\mathrm{c})$.  The methods used in
the present paper are based on a perturbation argument around
$\beta=0$ and their extension to the whole subcritical regime does not
appear feasible.  In principle, the techniques developed in
\cite{CM,LY}, which require as an input a strong spatial mixing of the
stationary probability, can be applied up to the critical temperature.
Those techniques have been however developed for short range
interactions and they do not seem, at least directly, applicable to
mean field Hamiltonians.

Another, somehow more fundamental, issue is whether $1/N^3$ is the
right scaling of the spectral gap in the supercritical regime.  
We expect that this is the correct scaling for
$\beta$ between $\beta_\mathrm{c}$ and $2\beta_\mathrm{c}$.
We mention that this behavior is also the one expected for the
Kawasaki dynamics for the low temperature two dimensional Ising model
with plus boundary condition (pure state) \cite{CCM}. Indeed, in this
case the heuristic picture presented in the Introduction corresponds
to the diffusion of the Wulff bubble.  While the statement (ii) in
Theorem~\ref{t:mt} is proven by exhibiting a suitable slowly varying
test function, a proof of a matching lower bound appears considerably
harder. The ABC model is however much simpler than short range models
and it therefore might be a useful starting point toward the
understanding of conservative dynamics in the phase transition region.

For $\beta$ larger than $2\beta_\mathrm{c}$, a preliminary question is
whether the other critical points of $\mc F_\beta$ correspond to local
minima. In such a case it is possible to construct a slowly
varying test function which yields the upper bound $\gap(L_N^\beta)
\le \exp\{ - c N \}$ for some constant $c=c(\beta)>0$.  Observe that
the general argument in \cite{CCM} gives for free the lower bound
$\gap(L_N^\beta) \ge \exp\{-C N\}$ for some $C=C(\beta)<+\infty$.

We finally discuss the behavior of the spectral gap of the ABC process
on an interval with zero flux condition at the endpoints.  As shown in
\cite{Ayetal}, in such a case the process is reversible with respect
to a mean field Gibbs probability for all values of the densities. In
the high temperature regime $\beta\ll 1$, the methods here developed
can be directly applied to get the diffusive behavior $1/N^2$.  
As far as the low temperature regime is concerned, the case of equal
densities is the same as the one on the ring and we can therefore
conclude that the upper bound $1/N^3$ holds also in this setting.

\section{Asymptotics of the Gibbs measure}

The upper bound on the spectral gap in the supercritical regime
requires the law of large numbers for the empirical density with
respect to the Gibbs measure $\nu_N^\beta$.  This result is proven by
combining the large deviations principle for $\nu_N^\beta$ with the
analysis of the minimizers of the free energy in \cite{Ayetal}.
As $\nu_N^\beta$ is a Gibbs measure with a mean field interaction, 
the associated large deviations principle can be proven by standard
tools. The specific application to the ABC model has not however been
detailed in the literature, we thus present here the whole argument.

\subsection*{Empirical density}

We let $\mb T :=\mb R / \mb Z$ be the one-dimensional torus of side
length one; the coordinate on $\mb T$ is denoted by $r \in [0,1)$.
The inner product in $L^2(\mb T,dr;\mb R^3)$ is denoted by
$\langle\cdot,\cdot\rangle$.
We set $\tilde{\mc M} := L^\infty \big(\mb T, dr;[0,1]^3 \big)$ and
denote by $\rho=(\rho_A,\rho_B,\rho_C)$ its elements.  We consider
$\tilde{\mc M}$ endowed with the weak* topology. Namely, a sequence
$\{\rho^n\}$ converges to $\rho$ in $\tilde{\mc M}$ iff
$\langle \rho^n, \phi \rangle \to \langle \rho, \phi \rangle$ for any
function $\phi\in L^1(\mb T,dr;\mb R^3)$, equivalently for any smooth
function $\phi\in C^\infty(\mb T;\mb R^3)$. Note that $\tilde{\mc M}$
is a compact Polish space, i.e.\ separable, metrizable, and complete.

We introduce
\begin{equation}
  \label{mcm}
  \mc M := \Big\{\rho\in \tilde{\mc M} \,:\: \rho_A+\rho_B+\rho_C=1\,,\: 
  \int_0^1\!dr\,\rho_\alpha(r)= \frac 13\,,~\alpha\in\{A,B,C\}
  \Big\}
\end{equation}
noticing it is a closed subset of $\tilde{\mc M}$ that we
consider equipped with the relative topology and the associated Borel
$\sigma$-algebra. 
The set of Borel probability measures on $\mc M$, denoted by $\ms P
(\mc M)$, is endowed with the topology induced by the weak convergence
of probability measures; namely, $\mc P_n\to \mc P$ iff for each
continuous $F\colon\mc M\to \mb R$ we have $\int \!d\mc P_n \, F
\to \int \!d\mc P \,F$. Note that also $\ms P (\mc M)$ is a compact Polish
space.

We define the \emph{empirical density} as the map
$\pi_N\colon\Omega_N\to\mc M$ given by
\begin{equation}
  \label{2:ed}
  \pi_N (\zeta) \, (r) \;:=\; \sum_{x\in\mb Z_N} \eta (x) \, 
  \ind_{[ x/N, (x+1)/N )}(r) \,,\quad r\in \mb T, 
\end{equation}
recall $\eta$ is the map defined at the beginning of
Section~\ref{s:nr}.  We set $\mc P_N^\beta:= \nu_N^\beta\circ
\pi_N^{-1}$ namely, $\mc P_N^\beta$ is the law of $\pi_N$ when $\zeta$
is distributed according to $\nu_N^\beta$. Note that $\{\mc
P_N^\beta\}$ is a sequence in $\ms P (\mc M)$.

\subsection*{Large deviations principle}

The \emph{entropy} (with a sign convention opposite to the standard
one in physical literature) is the convex lower semicontinuous
functional $\mc S\colon\mc M \to [0,+\infty)$ defined by
\begin{equation}
  \label{en}
  \mc S(\rho) := \int_0^1\!dr\:
  \Big[ \rho_A(r) \log \frac {\rho_A(r)}{1/3}  + \rho_B(r) \log
  \frac{\rho_B(r)}{1/3} 
    + \rho_C(r) \log \frac {\rho_C(r)}{1/3} \Big] 
\end{equation}
and the \emph{energy} is the continuous functional 
$\mc H\colon\mc M \to \mb R$  defined by  
\begin{equation}
  \label{Ham}
  \mc H(\rho) := \int_0^1\!dr \int_r^1\,dr'\:
  \Big[ \rho_A(r) \rho_C(r') +\rho_B(r) \rho_A(r') +\rho_C(r) \rho_B(r') \Big].
\end{equation}
For $\beta\ge 0$ the \emph{free energy} (in which we omit the
prefactor $1/\beta$) is finally the 
functional $\mc F_\beta\colon\mc M \to \mb R$ defined by
\begin{equation}
  \label{fe}
  \mc F_\beta := \mc S + \beta\, \mc H .
\end{equation}

\begin{theorem}
  \label{t:ld}
  The sequence $\{\mc P_{N}^\beta \}$ satisfies a large deviation
  principle with rate function 
  $\mc I_\beta= \mc F_\beta -\inf\mc F_\beta$.
  Namely, for each closed set $\mc C \subset \mc
  M$ and each open set $\mc O \subset \mc M$
  \begin{eqnarray*}
    &&
    \varlimsup_{N\to\infty} \frac 1{N} \log \mc P_{N}^\beta
    \big( \mc C \big)
    \;\leq\; - \inf_{\rho \in \mc C} \mc I_\beta (\rho)
    \\
    &&
    \varliminf_{N\to\infty} \frac 1{N} \log \mc P_{N}^\beta
    \big( \mc O \big) \;\geq\; -  \inf_{\rho \in \mc O} \mc I_\beta(\rho).
  \end{eqnarray*}
\end{theorem}

Since the beautiful Lanford's lectures \cite{La}, large deviations
principles for Gibbs measures has become a basic topic in equilibrium
statistical mechanics, see in particular \cite{El} for the case of
mean field interactions. On the other hand, the current setting is not
completely standard as we are looking to large deviations of the empirical
density for canonical Gibbs measures. We therefore give a detailed
proof of the above result. 
The first step is the large deviations principle when $\beta=0$;
recall that $\mc P_N^0= \nu_N^0\circ \pi_N^{-1}$ is the law of $\pi_N$
when $\zeta$ is distributed according to $\nu_N^0$ which is the
uniform probability on $\Omega_N$.

\begin{lemma}
  \label{t:ld0}
  The sequence $\{ \mc P_{N}^0\}$ satisfies a large deviation
  principle with rate function $\mc S$.  Namely, for
  each closed set $\mc C \subset \mc M$ and each open set $\mc O
  \subset \mc M$
  \begin{eqnarray}
    &&
    \label{ub0}
    \varlimsup_{N\to\infty} \frac 1{N} \log \mc P_{N}^0
    \big( \mc C \big)
    \;\leq\; - \inf_{\rho \in \mc C} \mc S (\rho)
    \\
    &&
    \label{lb0}
    \varliminf_{N\to\infty} \frac 1{N} \log \mc P_{N}^0
    \big( \mc O \big) \;\geq\; -  \inf_{\rho \in \mc O} \mc S(\rho).
  \end{eqnarray}
\end{lemma}

\begin{proof}
The proof is split in few steps.

\smallskip
\noindent\emph{Step 1.}
Set 
\begin{equation}
  \label{AL}
  \mc A :=\Big\{\lambda \in  C^1(\mb T;\mb R^3) : 
  \int_0^1\!dr \, \frac {e^{\lambda_\alpha(r)}}{e^{\lambda_A(r)}
      +e^{\lambda_B(r)} +e^{\lambda_C(r)}} = \frac 13,\,
    \alpha\in\{A,B,C\}\Big\} 
\end{equation}
and let $\Lambda\colon \mc A \to \mb R$ be the functional 
\begin{equation}
  \label{Lam}
  \Lambda(\lambda) :=  \int_0^1\!dr\,
  \log \Big[ \frac 13 
  \Big( e^{\lambda_A(r)} +e^{\lambda_B(r)} +e^{\lambda_C(r)}\Big) \Big];
\end{equation}
we shall prove that for each $\lambda\in\mc A$
\begin{equation}
  \label{press}
  \lim_{N\to \infty} \frac 1N \log 
  \int\!d\mc P_N^0 (\rho) \, \exp\big\{ N \langle\lambda,\rho \rangle\big\}
  = \Lambda(\lambda). 
\end{equation}

Denote by $\lambda^N(x)$ the average of $\lambda$ in the interval 
$\big[ x/N, (x+1)/N\big)$, 
\begin{equation*}
\lambda_\alpha^N(x):= N \int_{\frac xN}^{\frac {x+1}N}\!dr\,
\lambda_\alpha(r)\,,\qquad x\in\mb Z_N\,,\;\;\alpha\in \{A,B,C\}\;.   
\end{equation*}
From the very definition of the measure $\mc P_N^0$, 
\begin{equation}
  \label{s1}
  \begin{split}
     \int\!d\mc P_N^0 (\rho) \, 
   \exp\big\{ N \langle\lambda,\rho \rangle\big\} 
   & = \sum_{\zeta\in\Omega_N} \nu^0_N(\zeta) 
   \exp\big\{N \langle\lambda,\pi^N(\zeta) \rangle\big\}
   \\
   &= \sum_{\zeta\in\Omega_N} \nu^0_N(\zeta) 
   \prod_{x\in\mb Z_N} \exp\Big\{
   \sum_{\alpha\in\{A,B,C\}} \lambda^N_\alpha(x)\eta_\alpha(x) \Big\}.
  \end{split}
\end{equation}
We denote by $\mu_N^\lambda$ the product measure on
$\tilde{\Omega}_N=\{A,B,C\}^{\mb Z_N}$ with marginals
$\mu_{N,x}^\lambda$ given by
\begin{equation*}
  \mu_{N,x}^\lambda(\alpha) = \frac{e^{\lambda_\alpha^N(x)}}
    {e^{\lambda_A^N(x)}+e^{\lambda_B^N(x)}+e^{\lambda_C^N(x)}}
    \,,\qquad \alpha\in\{A,B,C\}.
\end{equation*}
When $\lambda=0$ we drop the superscript $\lambda$ from the notation
so that $\mu_N$ is the uniform measure on $\tilde{\Omega}_N$. 

Set $\Xi_N(\lambda) := \prod_{x\in\mb Z_N}
\Big( e^{\lambda_A^N(x)}+e^{\lambda_B^N(x)}+e^{\lambda_C^N(x)}\Big)$.
As $\nu_N^0 = \mu_N (\cdot\,|\Omega_N)$, from \eqref{s1} we
get
\begin{equation*}
  \int\!d\mc P_N^0 (\rho) \, 
   \exp\big\{ N \langle\lambda,\rho \rangle\big\} 
   = \frac{\Xi_N(\lambda)}{3^N} 
   \: \frac {\mu_N^\lambda(\Omega_N)}{\mu_N(\Omega_N)}.
\end{equation*}
We claim that for each $\lambda\in\mc A$ it holds $\lim_N \frac 1N
\log \mu_N^\lambda(\Omega_N) =0$.  The proof of this step is then
completed by observing that $\frac 1N \log \big[
\Xi_N(\lambda)/3^N]\to \Lambda(\lambda)$.

To prove the claim, we write
\begin{equation*}
  \mu_N^\lambda\big( \Omega_N\big) 
  = \mu_N^\lambda\Big( \frac 1N \sum_{x\in\mb Z_N} \eta_\alpha(x) =
  \frac 13,\, \alpha\in\{A,B,C\} \Big). 
\end{equation*}
In view of the smoothness of
$\lambda$  and the constraints in \eqref{AL}, for each
$\alpha\in\{A,B,C\}$ 
\begin{equation*}
  \begin{split}
  \frac 1N \sum_{x\in\mb Z_N} \mu_N^\lambda \big(\eta_\alpha(x) \big)
  &=\frac 1N \sum_{x\in\mb Z_N} \frac{ e^{\lambda_\alpha^N(x)} }
  {e^{\lambda_A^N(x)} + e^{\lambda_B^N(x)} + e^{\lambda_C^N(x)}}
  \\
  &
  = \int_0^1\!dr\, \frac{ e^{\lambda_\alpha(r)} }
  {e^{\lambda_A(r)} + e^{\lambda_B(r)} + e^{\lambda_C(r)}}
  + O\Big(\frac 1{{N}}\Big)
  = \frac 13 + O\Big(\frac 1{{N}}\Big).
  \end{split}
\end{equation*}
The claim now follows from an application of the local central limit
theorem for triangular arrays, see e.g.\ \cite[Ch.~VII]{petrov}.

\smallskip
\noindent\emph{Step 2.}
We here prove the large deviations upper bound \eqref{ub0}.
Given $\lambda\in\mc A$, let $\mc P_N^{0,\lambda}$ be the probability
on $\mc M$ defined by 
\begin{equation}
  \label{tilt}
  d \mc P_N^{0,\lambda} :=  
      \exp\big\{N 
        \big[\langle\lambda,\cdot\rangle -\Lambda_N(\lambda) \big] \big\}
            \, d \mc P_N^{0}
\end{equation}
where 
\begin{equation*}
  \Lambda_N(\lambda) =\frac 1N \log 
  \int\!d \mc P_N^{0} (\rho) \: e^{ N\langle\lambda,\rho\rangle}.
\end{equation*}
Given a measurable subset $\mc B$ of $\mc M$, we then have 
\begin{equation*}
  \mc P_N^0 (\mc B) = \int_{\mc B} \!d \mc P_N^{0,\lambda} 
  \: \frac{d \mc P_N^0}{ d \mc P_N^{0,\lambda}}
  \:\le\: \sup_{\rho\in\mc B} \exp\big\{ - N \big[
  \langle\lambda,\rho\rangle - \Lambda_N(\lambda)\big]\big\}.
\end{equation*}
In view of Step 1, $\Lambda_N(\lambda)\to \Lambda(\lambda)$ as
$N\to\infty$. We thus get
\begin{equation*}
  \varlimsup_{N\to\infty} \frac 1N \log\mc P_N^0 (\mc B) 
  \le - \inf_{\rho\in\mc B} 
  \big\{ \langle\lambda,\rho\rangle -\Lambda(\lambda) \big\}.
\end{equation*}

By optimizing with respect to $\lambda\in \mc A$ and using a mini-max
lemma, see e.g.\ Lemmata 3.2 and 3.3 in \cite[App.~2]{KL}, we deduce
that for each compact $\mc K \subset \mc M$
\begin{equation*}
  \label{ubB}
  \varlimsup_{N\to\infty} \frac 1N \log\mc P_N^0 (\mc K) 
  \le - \inf_{\rho\in\mc K} \, \sup_{\lambda\in\mc A} \,
  \big\{ \langle\lambda,\rho\rangle -\Lambda(\lambda) \big\} 
  =- \inf_{\rho\in\mc K} \mc S(\rho)
\end{equation*}
where the last identity follows by Legendre duality.
By the compactness of $\mc M$ this concludes the proof of the
upper bound.

\smallskip
\noindent\emph{Step 3.} 
Given two probability measures $P$ and $Q$, we denote by
$\mathrm{Ent}(Q|P)=\int\! dQ \, \log [dQ/dP]$ the \emph{relative
entropy} of $Q$ with respect to $P$.  
A simple computation based on Jensen inequality, see e.g.\
\cite[Prop.~4.1]{Je}, shows that the large deviations lower bound
\eqref{lb0} can be deduced from the following statement. For each
$\rho\in \mc M$ there exists a sequence of probability measures 
$\{\mc Q_N^{\rho}\}$ such that
\begin{equation}
  \label{entb}
  \mc Q_N^{\rho}\to\delta_\rho
  \qquad \textrm{and} \qquad
  \varlimsup_{N\to\infty} \frac 1N 
  \mathrm{Ent}\big( \mc Q_N^{\rho} \big| \mc P_N^{0} \big) \le \mc S(\rho).
\end{equation}

We here construct the sequence $\{\mc Q_N^{\rho}\}$ when $\rho$ is
continuously differentiable and bounded away from $0$ and $1$.
For such a $\rho$ let $\lambda=\lambda(\rho)$ be such that 
\begin{equation*}
  \rho_\alpha =
  \frac {e^{\lambda_\alpha}}{e^{\lambda_A} +e^{\lambda_B} +e^{\lambda_C}}
  \,,\qquad \alpha\in \{A,B,C\}.
\end{equation*}
Observe that $\lambda\in\mc A$ since $\rho$ is
continuously differentiable and  bounded away from $0$ and $1$. 
Recalling \eqref{tilt}, we claim that 
$\{\mc P_N^{0,\lambda(\rho)}\}$ fulfils the condition \eqref{entb}. 
The law of large numbers 
$\mc P_N^{0,\lambda(\rho)} \to \delta_\rho$ can be indeed checked by
the same computations of Step 1. 
Furthermore, in view of such law of large numbers and Step 1,  
\begin{equation*}
  \lim_{N\to\infty}\frac 1N
  \mathrm{Ent}\big( \mc P_N^{0,\lambda(\rho)} \big| \mc P_N^{0} \big) 
  = \langle \lambda,\rho \rangle -\Lambda(\lambda) = \mc S(\rho)
\end{equation*}
where the last equality follows from the choice of $\lambda$ and
Legendre duality.

\smallskip
\noindent\emph{Step 4.} 
The proof of the lower bound is here concluded by an approximation
argument \cite[Prop.~4.1]{Je}.  Let $\mc M_\circ$ be the subset of
$\mc M$ given by the continuously differentiable profiles bounded away
from $0$ and $1$.  The condition that a large deviation rate function
is lower semicontinuous is not restrictive.  More precisely, if a
sequence of probabilities satisfies the large deviations lower bound
for some rate function, then the lower bound still holds with the
lower semicontinuous envelope of such rate function. If we let $\mc
S^\circ$ be the functional equal to $\mc S$ on $\mc M_\circ$ and $\mc
S^\circ (\rho)=+\infty$ otherwise, in view of Step 3, the proof of the
lower bound \eqref{lb0} is concluded if we show that the lower
semicontinuous envelope of $\mc S^\circ$ is $\mc S$. This amounts to
prove that
\begin{equation*}
  \mc S(\rho) = 
  \sup_{ \mc O\ni \rho} \, \inf_{\hat\rho\in \mc O \cap \mc M_\circ} 
  \mc S(\hat\rho)
\end{equation*}
where the first supremum is carried over all the open neighborhoods of
$\rho$. The previous identity is easily proven by considering a
sequence $\{\rho^n\}$ of continuously differentiable profiles bounded
away from $0$ and $1$ which converges to $\rho$ a.e.\ in $\mb T$.
\end{proof}

In view of the continuity of the functional $\mc H$ on $\mc M$, 
the large deviations principle for the sequence $\{\mc P_N^\beta\}$ is
straightforward consequence of Lemma~\ref{t:ld0} and Laplace-Varadhan
theorem.

\begin{proof}[Proof of Theorem~\ref{t:ld}]
  Recalling definitions \eqref{ham}, \eqref{2:ed} and \eqref{Ham}, we
  claim that for each $\zeta\in\Omega_N$
  \begin{equation}
    \label{h=H}
    H_N (\zeta) =\mc H(\pi^N(\zeta)) \;.
  \end{equation}
  It is indeed enough to notice that by writing explicitly the right
  hand side above the diagonal terms vanish since
  $\eta_{A}(x)+\eta_{B}(x)+\eta_{C}(x) =1$, $x\in\mb Z_N$.

  Recall that $\nu_N^0$ is the uniform probability on $\Omega_N$ and
  let $\mc B$ be a measurable subset of $\mc M$.  From \eqref{h=H} and
  the definitions of the measures $\mc P_N^\beta$ and
  $\nu_N^\beta$, see \eqref{gm}
  \begin{equation*}
    \begin{split}
    \mc P_N^\beta(\mc B) & =
    \sum_{\substack{\zeta\in \Omega_N \\ \pi^N(\zeta)\in\mc B}}
    \nu_N^\beta(\zeta) 
    = \frac {|\Omega_N|}{Z_N^\beta}
    \sum_{\substack{\zeta\in \Omega_N \\ \pi^N(\zeta)\in\mc B}} 
    \nu^0_N(\zeta) \,  e^{-\beta N H_N (\zeta)} 
    \\
    &= \frac {|\Omega_N|}{Z_N^\beta} 
    \int_{\mc B} \!d \mc P_N^0\, e^{-\beta N \,\mc H (\pi_N)}.
    \end{split}
  \end{equation*}
  In particular,  by taking $\mc B =\mc M$,
   \begin{equation*}
    \frac{Z_N^\beta}{|\Omega_N|} 
      = \int\!d \mc P_N^0\, e^{-\beta N \,\mc H (\pi_N)}.
  \end{equation*}
  Since $\mc M$ is compact and $\mc H \colon \mc M \to \mb R$ is
  continuous, by using Lemma~\ref{t:ld0} and Laplace-Varadhan theorem,
  see e.g.\ \cite[Thm.~4.3.1]{DZ}, we deduce
  \begin{equation*}
    \lim_{N\to\infty}   \frac 1N \log \frac{Z_N^\beta}{|\Omega_N|} 
    = \sup_{\rho\in \mc M} \big\{ -\beta \mc H(\rho) -\mc S(\rho)\big\}
    = -\inf_{\rho\in\mc M} \mc F_\beta(\rho).  
  \end{equation*}
  Let $\mc C$ and $\mc O$ be respectively a closed and an open 
  subset of $\mc M$. Again from Lemma~\ref{t:ld0} and Laplace-Varadhan
  theorem, see e.g.\ \cite[Ex.~4.3.11]{DZ}, we deduce
  \begin{eqnarray*}
    &&
    \varlimsup_{N\to\infty} \frac 1{N} \log 
    \int_{\mc C} \!d \mc P_N^0\, e^{-\beta N \,\mc H (\pi^N)}
    \;\leq\; - \inf_{\rho \in \mc C} \mc F_\beta (\rho)
    \\
    &&
    \varliminf_{N\to\infty} \frac 1{N} \log 
    \int_{\mc O} \!d \mc P_N^0\, e^{-\beta N \,\mc H (\pi^N)}
    \;\geq\; - \inf_{\rho \in \mc O} \mc F_\beta (\rho).
  \end{eqnarray*}
  The theorem follows readily.
\end{proof}

\subsection*{Minimizers of the free energy}

We here recall the results in \cite{Ayetal} concerning the minimizers
of the free energy $\mc F_\beta$ in \eqref{fe} that are needed in our
analysis. As discussed in \cite{Ayetal}, the Euler-Lagrange equation 
$\delta \mc F_\beta=0$ can be, equivalently, written as the system of
ordinary differential equations 
\begin{equation}
  \label{ele}
  \begin{split}
    & \rho_A'  = \beta\, \rho_A (\rho_C - \rho_B )\\
    & \rho_B' = \beta\, \rho_B (\rho_A - \rho_C ) \\
    & \rho_C'  = \beta\, \rho_C (\rho_B - \rho_A ).
  \end{split}
\end{equation}
Note that the above condition is equivalent to the statement that
$\rho=(\rho_A,\rho_B,\rho_C)$ is a stationary solution to the
hydrodynamic equation \eqref{hy}. 

Let $\beta_\mathrm{c}:= 2\pi \sqrt{3}$.  In \cite{Ayetal} it is proven
that for $\beta\in[0,\beta_\mathrm{c}]$ the unique solution to
\eqref{ele} in $\mc M$ is the homogeneous profile $\upbar{\rho} := \big(
\tfrac 13, \tfrac 13,\tfrac 13 \big)$. On the other hand, when
$\beta>\beta_\mathrm{c}$ there are non trivial solutions. In
particular, there exists a unique $\rho\in C^\infty\big(\mb R;\mb
R^3\big)$ satisfying the following conditions: (i) $\rho$ solves
\eqref{ele}, (ii) $\rho$ is periodic with period $1$, (iii) $\rho$
satisfies the constraints in \eqref{mcm} and can therefore be thought
as an element in $\mc M$, (iv) the center of mass of the $B$ species
is $1/2$, i.e. $ 3 \, \int_0^1\!dr \: r \rho_B(r) =1/2$.  We shall
denote this solution by $\rho^*=\rho^{*,\beta}$. Note that any
translation of $\rho^{*,\beta}$ satisfies conditions (i)--(iii) above
but not (iv). We emphasize that the condition (ii) requires the
minimal period to be one. Indeed, as discussed in \cite{Ayetal}, when
$\beta> n\beta_\mathrm{c}$ for some integer $n\ge 2$, there are solutions of
\eqref{ele} with period $1/n$. These are the other critical points of
$\mc F_\beta$ which may lead to a metastable behavior.

Given $s\in \mb T$ we denote by $\tau_s:\mc M\to \mc M$ the
translation by $s$, namely $(\tau_s \rho)(r) = \rho(r-s)$.  If $\mc P$
is a probability on $\mc M$, the corresponding translation is $\mc
P\circ \tau_s^{-1}$. The following statement is a (partial) rewriting
of Theorems~4.1 and 5.2 in \cite{Ayetal}. 

\begin{theorem} 
  \label{t:mfe}
  $~$
  \begin{itemize}
  \item[(i)] If $\beta\in [0,\beta_\mathrm{c}]$ then
    \begin{equation*}
      \arginf \mc F_\beta = \big\{ \upbar{\rho} \big\};
    \end{equation*}
    namely, the unique minimizer of $\mc F_\beta$ is $\upbar{\rho}$.
  \item[(ii)] If $\beta\in (\beta_\mathrm{c},+\infty)$ then
    \begin{equation*}
      \arginf \mc F_\beta = \big\{ \tau_s \rho^{*,\beta} 
      \,,\: s\in \mb T \big\};
    \end{equation*}
   namely, $\mc F_\beta$ has a one-parameter family of minimizers which are
   obtained by translating $\rho^{*,\beta}$. 
  \end{itemize}
\end{theorem}

\subsection*{Law of large numbers for the empirical density}

As a corollary of the previous statements, we here prove the law of
large numbers for the sequence $\{\mc P_N^\beta\}$. The corresponding
limit point charges the set of minimizers of the free energy only. 
In the supercritical case we show that each $\tau_s\rho^{*,\beta}$,
$s\in\mb T$, is chosen with uniform probability.

\begin{theorem}$~$
\label{t:lln}
  \begin{itemize}
  \item[(i)] If $\beta\in [0,\beta_\mathrm{c}]$ then the sequence
    $\{\mc P_N^\beta\}$ converges to
    $\delta_{\upbar{\rho}}$.
  \item[(ii)] 
    If $\beta\in (\beta_\mathrm{c},+\infty)$ then the sequence $\{\mc
    P_N^\beta\}$ converges to $\int_0^1\!ds \: \delta_{\tau_s \rho^{*,\beta}} $.
  \end{itemize}
\end{theorem}

\begin{proof}
  Item (i) follows immediately from the large deviations principle
  stated in Theorem~\ref{t:ld} and the uniqueness of minimizers of
  $\mc F_\beta$ stated in item (i) of Theorem~\ref{t:mfe}.

  To prove item (ii), let $\vartheta:\Omega_N\to\Omega_N$ be the
  microscopic translation, i.e.\ $\vartheta \zeta$ is the
  configuration defined by $(\vartheta \zeta)(x) =\zeta(x-1)$,
  $x\in\mb Z_N$.  As follows from definition \eqref{gm}, the
  probability $\nu_N^\beta$ is translation invariant, i.e.\
  $\nu_N^\beta\circ \vartheta^{-1} = \nu_N^\beta$. This implies that
  the probability $\mc P_N^\beta$ is invariant by discrete
  translations: $\mc P_N^\beta \circ \tau_{x/N}^{-1}= \mc P_N^\beta$,
  $x\in \mb Z_N$.  By the compactness of $\mc M$, there exists a
  probability $\mc P\in\ms P(\mc M)$ and a subsequence $\{\mc
  P^\beta_N\}$ such that $\mc P_{N}^\beta\to \mc P$.  We claim that
  $\mc P$ is translation invariant. Indeed, fix a continuous function
  $F$ on $\mc M$ and $s\in \mb T$. Observe that, in view of the
  compactness of $\mc M$, $F$ is uniformly continuous.  Pick now a
  sequence $\{x_{N}\in \mb Z_{N}\}$ such that $x_{N}/N\to s$. The
  uniform continuity of $F$ implies that $\tau_{x_{N}/N} F$ converges
  uniformly to $\tau_s F$.  Since $\int\! d\mc P_N^\beta \,
  \tau_{x_N/N} F =\int\! d\mc P_N^\beta \, F$, by taking the limit
  $N\to \infty$ we deduce that $\int\! d\mc P \, \tau_s F =\int\! d\mc
  P \, F$. In view of the arbitrariness of $F$ we conclude $\mc P\circ
  \tau_s^{-1} =\mc P$.  Moreover, Theorem~\ref{t:ld} and item (ii) in
  Theorem~\ref{t:mfe} imply that the support of $\mc P$ is a subset of
  $\big\{ \tau_s \rho^{*,\beta} \, ,\: s\in \mb T \big\}=:\mc T$. Let
  now $\phi\colon \mc T \to \mb T$ be the bijection defined by
  $\tau_s\rho^{*,\beta}\mapsto s$ and set $\lambda:= \mc P
  \circ\phi^{-1}$. Since $\mc P= \mc P \circ \tau_s^{-1}$, $s\in\mb
  T$, we deduce that $\lambda$ is a translation invariant probability
  measure on $\mb T$. As the Lebesgue measure $dr$ is the unique
  translation invariant probability measure on $\mb T$ we deduce
  $\lambda(dr)=dr$. The proof is now completed by observing that for
  each continuous $F\colon \mc M \to \mb R$ the previous identity
  imply $\int\!d\mc P (\rho) \, F(\rho) = \int_0^1\!ds\,
  F(\tau_s\rho^*)$.
\end{proof}

\section{Lower bound on the spectral gap in the subcritical case}

In this section we prove the first statement in Theorem~\ref{t:mt}.
This result is derived from an analysis of a perturbed interchange
process, that is detailed in Appendix~\ref{s:a}, and a comparison of
the corresponding Dirichlet forms. This method has been introduced in
\cite{Q} and applied in different contexts, see e.g.\ \cite{bcdp}. 

We start by defining the ABC process on the complete graph with $N$
vertices.  Given a (unoriented) bond $\{x,y\}\subset \mb Z_N$, $x\neq
y$, and a function $f\colon \Omega_N \to \mb R$, we introduce the
gradient
\begin{equation}
  \label{nabla}
   \big(\nabla_{x,y} f \big) \, (\zeta) := f(\zeta^{x,y}) -f (\zeta)
\end{equation}
where, as in \eqref{ex}, $\zeta^{x,y}$ denotes the configuration obtained from
$\zeta$ exchanging the particles in $x$ and $y$. The ABC dynamics on
the complete graph is then defined by the Markov generator
\begin{equation}
\label{Lc=}
  \ms L_N^\beta f  :=  \sum_{\{x,y\}\subset \mb Z_N} c^\beta_{x,y} 
   \nabla_{x,y} f 
\end{equation}
where, recalling \eqref{ham}, the jump rates
$c^\beta_{x,y}=c_{x,y}^{\beta,N} \colon \Omega_N\to (0,+\infty)$ are given by 
\begin{equation}
  \label{ratlc}
  c_{x,y}^{\beta,N} := \frac 1N
  \exp\Big\{ -\frac{\beta N}{2} \, \nabla_{x,y} H_N \Big\}.
\end{equation}
In particular, the above rates satisfy the detailed balance with
respect to the probability measure $\nu_N^\beta$ defined in \eqref{gm}.
Recalling \eqref{rates} and \eqref{ham}, we also observe that
$N\, c_{x,x+1}^{\beta,N}=  c_x^{\beta,N}$.  

In Appendix~\ref{s:a} we prove that, provided $\beta$ is small
enough, the spectral gap of $\ms L_N^\beta$ is of order one uniformly
in $N$.

\begin{lemma}
  \label{t:cor}
  There exist constants $\beta_0,C_1 \in (0,+\infty)$ such that for any
  $\beta\in [0,\beta_0]$ and any $N\ge 3$
  \begin{equation*}
    \gap\big( \ms L_{N}^\beta \big)  \ge  
    \frac {1}{C_1} \cdot 
  \end{equation*}
\end{lemma}

We denote by $D_N^\beta$ and $\ms D_N^\beta$ the Dirichlet forms
associated to the generators $L_N^\beta$ and $\ms L_N^\beta$,
respectively. That is, given $f\colon \Omega_N\to \mb R$, 
\begin{align}
  \label{forma}
  D_N^\beta(f,f) &:= \nu_N^\beta \big(f \, (-L_N^\beta) f\big) 
  = \frac 12 \sum_{x=1}^{N}
  \nu_N^\beta \big( c^\beta_x \, \big[ \nabla_{x,x+1} f \big]^2 \big),
  \\
  \label{formastorta}
  \ms D_N^\beta(f,f) &:= \nu_N^\beta \big(f \, (-\ms L_N^\beta) f\big) 
  = \frac 1{2} \sum_{\{x,y\}\subset\mb Z_N}
  \nu_N^\beta \big( c^\beta_{x,y} \, \big[ \nabla_{x,y} f \big]^2 \big).
\end{align}

\begin{lemma}
  \label{t:cam}
  The inequality
  \begin{equation*}
     \ms D_N^\beta(f,f) \le 2 \,e^{3\beta} \, N^2\, D_N^\beta(f,f)
  \end{equation*}
  holds for any $\beta\ge 0$ and any function $f\colon \Omega_N\to \mb R$.
\end{lemma}

\begin{proof}
  Given $\{x,y\}\subset \mb Z_N$ we let $T_{x,y}\colon \Omega_N\to
  \Omega_N$ be the involution defined by $T_{x,y}\zeta :=
  \zeta^{x,y}$. We use the same notation for the corresponding linear
  map on the set of functions $f\colon \Omega_N\to \mb R$, i.e.\ 
  $\big(T_{x,y} f\big) (\zeta) := f(T_{x,y}\zeta)=f\big(\zeta^{x,y}\big)$. 
  As it is simple to check, the long jump $T_{x,y}$ can be decomposed
  in terms of nearest neighbor jumps as follows
  \begin{equation*}
    T_{x,y} = T_{x+1,x} \, T_{x+2,x+1}\, \cdots \, T_{y-1,y-2} \, T_{y-1,y}
    \, T_{y-2,y-1}\, \cdots \, T_{x+1,x+2} \,T_{x,x+1}.
  \end{equation*}
  We then write the telescopic sum 
  \begin{equation*}
    \begin{split}
      &
      T_{x,y} f - f 
      \\
      &\quad 
      =
      \big[ 
      T_{x,x+1}  \cdots T_{y,y-1}  \cdots T_{x+2,x+1} T_{x+1,x} \, f
      \:-\: 
      T_{x,x+1} \cdots T_{y,y-1} \cdots T_{x+2,x+1} \, f
      \big]
      \\
      &\quad\; +
      \big[ 
      T_{x,x+1} \cdots T_{y,y-1} \cdots T_{x+3,x+2} T_{x+2,x+1} \, f
      \:-\: 
      T_{x,x+1} \cdots T_{y,y-1} \cdots T_{x+3,x+2} \, f
      \big]
      \\
      &\quad\; +\:\cdots\:+
      \big[ 
      T_{x,x+1} \cdots T_{y-2,y-1} T_{y,y-1} \, f
      \:-\: 
      T_{x,x+1} \cdots T_{y-2,y-1}\, f
      \big]
      \\
      &\quad\; 
      +\: 
      \cdots
      +\:
      \big[ T_{x,x+1} f - f\big].
    \end{split}
  \end{equation*}
  Whence 
  \begin{equation}
    \label{tele}
    \begin{split}
      &\nabla_{x,y} f (\zeta)      
      =
      \big(\nabla_{x,x+1} f\big)
      \big( T_{x+1,x} \cdots T_{y-1,y}\cdots T_{x+1,x+2}\zeta\big)
      \\
      &\qquad
      +
      \big(\nabla_{x+1,x+2} f\big)
      \big( T_{x+1,x} \cdots T_{y-1,y}\cdots T_{x+2,x+3}\zeta\big)
      \\
      &\qquad
      +\: 
      \cdots
      +\:
      \big(\nabla_{y-1,y} f\big)
      \big( T_{x+1,x} \cdots T_{y-2,y-1}\zeta\big)
      +\: 
      \cdots
      +\:
      \big(\nabla_{x,x+1} f\big) (\zeta).
    \end{split}
  \end{equation}

  In view of \eqref{ham} and \eqref{gm}, for any $\beta\in\mb R_+$, 
  any $z\in \mb Z_N$, and any positive function $g\colon \Omega_N\to \mb R_+$ 
  \begin{equation}
    \label{jac}
    \nu_N^\beta \big( T_{z,z+1} g \big) \le 
    \exp\big\{ N\beta \|\nabla_{z,z+1} H_N\|_\infty\big\} 
    \,  \nu_N^\beta ( g )
    \le e^{\beta/N} \,  
    \nu_N^\beta (g)
  \end{equation}
  where $\|\cdot\|_\infty$ denotes the uniform norm. 
  By Schwarz inequality in \eqref{tele} and using recursively the
  previous estimate we then get, for $1\le x<y\le N$ 
  \begin{equation}
    \label{sw}
    \nu_N^\beta\big( \big[\nabla_{x,y} f\big]^2 \big) \le 
    2\, [ 2 (y-x)-1 ] \, e^{2\beta} 
    \sum_{z=x}^{y-1} \nu_N^\beta\big( \big[\nabla_{z,z+1} f\big]^2 \big).
  \end{equation}
  Indeed, in the generic term on the right hand side of
  \eqref{tele} there is the composition of nearest neighbor
  exchanges $T_{x+k,x+k+1}$ whose number is at most $2(y-x)-2 \le 2
  N$.  In view of \eqref{jac} this yields the factor $e^{2\beta}$.  As
  the number of terms on the right hand side of \eqref{tele} is $2
  (y-x)-1$ and each bond $\{z,z+1\}$, $z=x,\cdots,y-1$ is used at most
  two times, the bound \eqref{sw} follows. 

  To conclude the proof of the lemma it is now enough to observe that
  the jump rates in the Dirichlet forms \eqref{forma} and
  \eqref{formastorta} respectively satisfy the bounds $c_x^{\beta}\ge
  e^{-\beta/(2N)}$ and $N c_{x,y}^{\beta} \le e^{\beta/2}$. In view of
  \eqref{sw} elementary computations now yield the statement.
\end{proof}

\begin{proof}[Proof of Theorem~\ref{t:mt}, item (i).]
  Recall the Rayleigh-Ritz variational characterization of the
  spectral gap \eqref{rr}. By Lemmata \ref{t:cor} and \ref{t:cam}
  we then deduce the statement with $\beta_0$ as in Lemma~\ref{t:cor} and
  $C_0 = \frac 12 e^{-3\beta_0} C_1$.
\end{proof}

\section{Upper bound on the spectral gap in the supercritical case}

We discuss here the upper bound on the spectral gap when
$\beta>\beta_\mathrm{c}$. In view of the Rayleigh-Ritz 
variational characterization \eqref{rr}, the proof will be
achieved by exhibiting a suitable test function. 
The naive picture is the following. When $\beta>\beta_\mathrm{c}$ and
$N$ is large, the ABC process essentially performs a random walk on the
set of minimizers of the free energy $\mc F_\beta$, which in
the supercritical case is homeomorphic to the one-dimensional
torus. We thus choose as test function the one that corresponds to the
slow mode of such random walk and conclude the argument.

\begin{proof}[Proof of Theorem~\ref{t:mt}, item (ii).]
Pick a Lipschitz function $\phi\colon\mb T\to\mb R$ such that
$\int_0^1\!dr\, \phi(r)=0$ to be chosen later and  
let $f_N\colon \Omega_N \to\mb R$ be the function 
\begin{displaymath}
  f_N :=
  \frac{1}{N}\sum_{x=1}^N\eta_B(x) \phi\big(\tfrac{x}{N}\big).
\end{displaymath}
By the Rayleigh-Ritz principle,
\begin{equation}
  \label{urr}
  \gap(L_N^\beta) \le 
  \frac{D_N^\beta(f_N,f_N)}{\nu_N^\beta \big(f_N ,f_N\big)}
\end{equation}
where the Dirichlet form $D_N^\beta$ has been defined in 
\eqref{forma}. We next estimate from below the denominator and from
above the numerator. 

To bound the variance of $f_N$, we first observe that, since $\phi$
has mean zero, we have  $\lim_{N}\, \nu_N^\beta \big(f_N\big)=
\frac 13 \int_0^1\!dr\, \phi(r)=0$.  
Recall \eqref{mcm} and let $F\colon\mc M \to \mb R$ be defined by 
\begin{equation*}
  F (\rho) = \int_0^1\!dr\, \rho_B(r) \,\phi(r).
\end{equation*}
The continuity of $\phi$ implies  
\begin{equation*}
  \lim_{N\to\infty}\sup_{\zeta\in\Omega_N} \:
  \big|  f_N(\zeta) - F\big(\pi_N(\zeta)\big) \big| =0,
\end{equation*}
where the empirical density $\pi_N\colon \Omega_N\to \mc M$
has been defined in \eqref{2:ed}. 
By assumption, $\beta>\beta_\mathrm{c}$ and therefore
Theorem~\ref{t:lln}, item (ii) implies  
\begin{equation*}
  \lim_{N\to \infty} \nu_N^\beta\big( f_N^2\big) =  
  \lim_{N\to \infty} \nu_N^\beta\big( F(\pi_N)^2\big)=
  \int_0^1\!ds \, \Big[ \int_0^1\!dr\, \rho^{*,\beta}_B(r-s)\phi(r) \Big]^2.
\end{equation*}
We can choose $\phi$ such that the right hand side above is strictly
positive. It is indeed enough to observe that, since $\rho^{*,\beta}$
is not constant,  
there exists a Lipschitz, mean zero, function $\phi$ such that 
$\int_0^1\!dr\, \rho^{*,\beta}_B(r)\phi(r)\neq 0$. 
For such choice of $\phi$ we deduce there exists a constant
$C(\beta)\in (0,+\infty)$ such that  $\nu_N^\beta \big(f_N,f_N\big)\ge
C(\beta)$ for any $N\ge 3$.  

We next bound the Dirichlet form. A straightforward computation yields 
\begin{equation*}
  \nabla_{x,x+1}f_N 
  = \frac{1}{N} 
  \Big[ \phi\big(\tfrac {x+1}N\big) -\phi\big(\tfrac {x}N\big)\Big]
  \, \big[ \eta_B(x) -\eta_B(x+1)\big] .
\end{equation*}
Since $c_x^{\beta,N} \le \exp\big\{\frac{\beta}{2N}\big\}$, we then get 
\begin{equation*}
  D_N^\beta(f_N,f_N) \le
  \, \exp\big\{\tfrac{\beta}{2N}\big\}
  \, \frac{1}{2N^2}  
  \sum_{x=1}^N 
  \Big[\phi\big(\tfrac {x+1}N\big) -\phi\big(\tfrac{x}N\big)\Big]^2 .
\end{equation*}
Therefore, letting $C_\phi$ be the Lipschitz constant of $\phi$,
\begin{equation*}
  \varlimsup_{N\to\infty} N^3 \, D_N^\beta(f_N,f_N) 
  \le \frac 12  \, C_\phi^2
\end{equation*}
which concludes the proof.
\end{proof}

\appendix
\section{Spectral gap for perturbed interchange processes}
\label{s:a}

We prove here a general result on the spectral gap on suitable
Markov chains on the set of permutations of $\{1,\dots,N\}$. 
The jumps of this chain are obtained by  randomly choosing a
transposition. As reference process we consider the so-called 
interchange process on $\{1,\dots,N\}$, see \cite{CLR,DS}. 
This process can be realized as the simple random walk on the graph
with vertex set given by the symmetric group $S_N$ and edges given by
the pairs $(\sigma_1,\sigma_2)\in S_N\times S_N$ such that the
composition $\sigma_1^{-1}\circ \sigma_2$ is a transposition. 
The reference invariant measure is thus the uniform
probability on the symmetric group. We then perturb this measure
according to the standard Gibbs formalism and consider an associated
reversible chain. Under general conditions on the energy, we show that
-- at high enough temperature -- the relaxation time of perturbed
chain behaves, for large $N$, as the one of the reference random walk.
The ABC dynamics on the complete graph \eqref{Lc=} can be realized by
looking at the previous chain in a colorblind way, that is resolving
only $3$ out of the $N$ colors.

Let $V_N:=\{1,\dots,N\}$ and $B_N:=\{b\subset V_N: |b|=2\}$. The
complete graph on $N$ vertices is $G_N:=(V_N,B_N)$ and 
$S_N:=\{\sigma:V_N\to V_N, \text{bijective}\}$ is the set of
permutations on $V_N$.  For $\sigma\in S_N$, $\{x,y\}\in B_N$
let $\sigma^{\{x,y\}}\in S_N$ be the permutation obtained 
by composing $\sigma$ with the transposition which exchanges $x$ and
$y$ 
\begin{displaymath}
  \sigma^{\{x,y\}}(z):=
  \begin{cases}
    \sigma(y) & \text{if $z=x$} \\
    \sigma(x) & \text{if $z=y$} \\
    \sigma(z) & \text{otherwise.}
  \end{cases}
\end{displaymath}
Given an \emph{energy function} $E_N\colon S_N\to\mb R$ 
and  $\beta\geq 0$, we define the probability on $S_N$ by 
\begin{displaymath}
  \pi_N(\sigma)=\pi_N^\beta(\sigma):=
  \frac{1}{Z_N^\beta}\exp\big\{-\beta E_N(\sigma)\big\}.
\end{displaymath}
where $Z_N^\beta$ is the normalization constant.  
For $f\colon S_N\to \mb R$, $a\in B_N$ define
$f^{a}(\sigma):=f(\sigma^{a})$, $\nabla_{a}f:=f^{a}-f$
and the Markov generator $\mc G_{N} = \mc G_{N}^\beta$  by
\begin{equation}
  \label{genG}  
  \mc G_{N} f := \sum_{a\in B_N}c_a\nabla_af 
\end{equation}
where the transition rates are 
\begin{equation}
  \label{ratN}
  c_a=c_{a}^{\beta,N}:=
  \frac{1}{N}\exp\big\{-\tfrac{\beta}{2}\, \nabla_a E_N\big\}.
\end{equation}
The associated Markov chain satisfies the detailed balance 
with respect to the probability $\pi_N$, i.e.,
\begin{equation}
  \label{eq:db}
  \pi_N\big( c_a g \big) =\pi_N\big( c_a g^a \big).
\end{equation}

The operator $\mc G_{N}$ is selfadjoint in $L^2(S_N,d\pi_N)$, the
corresponding Dirichlet form is
\begin{displaymath}
  \mathcal{E}_N(f,f):=
  \pi_{N}\big(f (-\mc G_{N}) f \big)=
  \frac{1}{2}\sum_{a\in B_N}\pi_{N}\big[c_{a} \, (\nabla_a f)^2\big].
\end{displaymath}

We next show that for $\beta$ small enough the spectral gap of 
$\mc G_{N}^\beta$ is strictly positive uniformly in $N$.

\begin{theorem}
  \label{t:a1}
  Assume $\sup_N\sup_{a}\Vert \nabla_a E_N\Vert_\infty<+\infty$.
  Then there exist $\beta_0,K_0> 0$ such that for any
  $\beta\in[0,\beta_0]$ and any $N$
  \begin{equation}
    \label{gapG}
    \frac 1{K_0}\leq
    \gap\big(\mc G_{N}^\beta\big)
    \leq K_0.
  \end{equation}
\end{theorem}

\begin{proof}[Proof of the upper bound]
  In view of the variational characterization \eqref{rr} of the
  spectral gap $\mc G_{N}$, it is enough to exhibit a suitable test
  function. 
  We next show that by choosing $f(\sigma)=\ind_{\{1\}}(\sigma(1))$, the
  upper bound in \eqref{gapG} follows. 

  The variance of $f$ is $\pi_N(f,f)= \pi_N\big(\sigma(1)=1\big)
  \big[1-\pi_N\big(\sigma(1)=1\big)\big]$. 
  To compute the Dirichlet form, we first observe that
  $\nabla_a f_N $ vanishes if $1\not\in a$. On the other hand, if
  $a=\{1,y\}$ for some $y\in\{2,\dots,N\}$
  \begin{equation*}
    \big[ \nabla_{\{1,y\}} f_N (\sigma) \big]^2= 
    \ind_{\{1\}}(\sigma(1)) +\ind_{\{1\}}(\sigma(y)).
  \end{equation*} 
  Whence, in view of \eqref{ratN},
  \begin{equation*}
    \begin{split}
      \mathcal{E}_N(f,f) & =
      \frac{1}{2} \sum_{y}\pi_N\Big(  
        c_{\{1,y\}} \big[ \ind_{\{1\}} (\sigma(1))
        +\ind_{\{1\}}(\sigma(y)) \big] \Big)
      \\ 
      & \leq \exp\Big\{\frac{\beta}{2} \sup_{N,a}\Vert \nabla_a
        E_N\Vert_\infty\Big\}
      \, \pi_N\big(\sigma(1)=1\big).
    \end{split}
  \end{equation*}
  Therefore
    \begin{displaymath}
    \gap(\mc G_{N})\leq
    \frac{\exp\Big\{ \frac{\beta}{2}
      \sup_{N,a}\Vert \nabla_a E_N\Vert_\infty\Big\} 
      } {1-\pi_N\big(\sigma(1)=1\big)}\cdot 
  \end{displaymath}
  It remains to show that the denominator above is bounded away from
  0 uniformly in $N$. We claim that
  \begin{equation}
    \label{eq:uni}
    \frac{1}{N}\exp\Big\{ 
    -\beta\sup_{N,a} \Vert\nabla_a E_N \Vert_\infty\Big\} 
    \leq
    \pi_N\big(\sigma(1)=1\big)
    \leq
    \frac{1}{N}\exp\Big\{ \beta\sup_{N, a}\Vert \nabla_a E_N\Vert_\infty\Big\}.
  \end{equation}
  Indeed, fix $x\in V_N$ and observe 
  \begin{multline*}
    \pi_N\big(\sigma(1)=1\big)=
    \frac{1}{Z_N}\sum_\sigma\pi_N(\sigma)\ind_{\{1\}} (\sigma(1))
    = \frac{1}{Z_N}\sum_\sigma\pi_N(\sigma^{\{1,x\}})
    \ind_{\{1\}}(\sigma(x))
    \\
    =\frac{1}{Z_N}
    \sum_\sigma\pi_N(\sigma)
    \frac{\pi_N(\sigma^{\{1,x\}})}{\pi_N(\sigma)}\ind_{\{1\}}(\sigma(x))
    = \pi_N\Big( e^{-\beta \, \nabla_{\{1,x\}}E_N}\ind_{\{1\}}(\sigma(x))\Big).
  \end{multline*}
  This yields
  \begin{displaymath}
    \exp\Big\{-\beta\sup_{N,a}\Vert \nabla_a E_N\Vert_\infty\Big\}
    \leq
    \frac{\pi_N\big(\sigma(x)=1\big)}{\pi_N\big(\sigma(1)=1\big)}
    \leq
    \exp\big\{\beta\sup_{N,a}\Vert \nabla_a E_N\Vert_\infty\Big\}.
  \end{displaymath}
  Summing over $x\in V_N$ and observing that $\sum_{x}
  \pi_N\big(\sigma(x)=1\big)=1$ we get \eqref{eq:uni}. 
\end{proof} 

\begin{proof}[Proof of the lower bound]
  The proof is based on the $\Gamma_2$ approach, see e.g.\ \cite{BE}, as
  adapted to the context of interacting particles systems in
  \cite{bcdp}. The starting point is the observation, which follows
  from the spectral theorem, 
  that $\gap(\mc G_{N})$ is the largest constant $k$ such that for
  any $f\colon S_N\to\mb R$
  \begin{equation}
    \label{eq:gap2}
    \pi_N\big[ (\mc G_N f)^2\big] 
    \geq {k} \, \mc E_N (f,f) = \frac{k}{2}
    \sum_{a\in B_N}\pi_N\left[c_a(\nabla_af)^2\right].
  \end{equation}

  To prove the lower bound in \eqref{gapG} it is therefore enough
  to show there exist a constant $k$ independent of $N$ such that
  \eqref{eq:gap2} holds.  We proceed in two steps. We first show that
 \begin{equation}
    \label{eq:gapp1}
    \begin{split}
    &\pi_N\big[(\mc G_Nf)^2\big]
    \\
    &\qquad \geq
    \sum_{\substack{ a,b\in B_N \\ a\cap b\not=\emptyset}}
    \pi_N\big[c_a c_b \nabla_a f \nabla_bf\big] 
    +\frac{1}{2}
    \sum_{\substack{ a,b\in B_N \\ a\cap b=\emptyset}}
    \pi_N\Big[c_ac_b\Big(1-\frac{c_b^a}{c_b}\big)\nabla_a f\nabla_b f\Big].
    \end{split}
  \end{equation}
  Then we prove there exists a constant $k$ independent of $N$ such
  that 
  \begin{equation}
   \label{eq:gapp2}
   \begin{split}
     &\frac{k}{2}\sum_{a\in B_N}\pi_N\big[c_a(\nabla_af)^2\big]
     \\
     &\qquad\le 
     \sum_{\substack{ a,b\in B_N \\ a\cap b\not=\emptyset}}
     \pi_N\big[c_ac_b\nabla_af\nabla_bf\big]
    +\frac{1}{2}
    \sum_{\substack{ a,b\in B_N \\ a\cap b=\emptyset}}
    \pi_N\Big[c_ac_b\Big(1-\frac{c_b^a}{c_b}\Big)\nabla_a f\nabla_b
    f\Big].
   \end{split}
  \end{equation}

  While the inequality \eqref{eq:gapp1} can be obtained as a
  consequence of Corollary~2.3 and Proposition~2.4 in
  \cite{bcdp}, we next give a direct proof in the present setting.
  Observe that
  \begin{equation}
    \label{3pezzi} 
    \begin{split}
      \pi_N\big[(\mc G_{N}f)^2\big] &=
      \sum_{a,b\in B_N}
      \pi_N\big[c_ac_b\nabla_a f\nabla_b f\big]\\
      &= \sum_{a\cap b\not=\emptyset}
      \pi_N\big[c_ac_b\nabla_a f\nabla_b f\big]
      +\sum_{a\cap b=\emptyset}\pi_N\big[c_ac_b\nabla_a f\nabla_b f\big].
    \end{split}
  \end{equation}
  We rewrite the last term as
  \begin{equation*}
    \begin{split}
      &\sum_{a\cap b=\emptyset}\pi_N\big[c_ac_b\nabla_a f\nabla_b f\big]
      \\ &\qquad
      =\frac{1}{2}
      \sum_{a\cap b=\emptyset}
      \pi_N\Big[c_ac_b\Big(1-\frac{c_b^a}{c_b}\Big)
      \nabla_a f\nabla_b f\Big]
      + \frac{1}{2}\sum_{a\cap b=\emptyset}
      \pi_N\Big[c_ac_b\Big(1+\frac{c_b^a}{c_b}\Big)
      \nabla_a f\nabla_b f\Big].
   \end{split}
  \end{equation*}
  We claim that the last term on the right hand side above 
  is positive. This statement together with \eqref{3pezzi} 
  trivially implies \eqref{eq:gapp1}. 
  To prove the previous claim, fix $a,b\in B_N$, with $a\cap
  b=\emptyset$ and observe that in this case 
  $(\nabla_a f)^a(\nabla_b f)^a=-\nabla_a
  f\nabla_b f^a$. 
  The detailed balance condition \eqref{eq:db} now implies
  \begin{equation*}
    \begin{split}
      &\pi_N\Big[c_ac_b\Big(1+\frac{c_b^a}{c_b}\Big)
      \nabla_a f\nabla_b f\Big]
      = \pi_N\Big[c_a\Big(c_b+c_b^a\Big)\nabla_a f\nabla_b f\Big]
      \\
      &\qquad=
      \pi_N\big[c_a\big(c_b^a+c_b\big) (\nabla_a f)^a (\nabla_b f)^a\big]
      =- \pi_N \big[c_a\big(c_b^a+c_b\big)
      \nabla_a f\nabla_b f^a\big].
    \end{split}
  \end{equation*}
  Then
  \begin{equation}
    \label{eq:1}
    \begin{split}
      \pi_N\Big[c_ac_b\Big(1+\frac{c_b^a}{c_b}\Big)
      \nabla_a f\nabla_b f\Big]
      & =
      \frac{1}{2}\, \pi_N\Big[c_a\Big(c_b+c_b^a\Big)
      \big( \nabla_a f\nabla_b f -\nabla_a f\nabla_b f^a\big)\Big]
      \\
      & =-\frac{1}{2}\, \pi_N\big[c_a\big(c_b+c_b^a\big)
      \nabla_a f\nabla_a\nabla_b f \big],
    \end{split}
  \end{equation}
  Furthermore, by direct computation,
  \begin{equation}
    \label{eq:sim}
    c_a\big(c_b+c_b^a\big)=c_b\big(c_a+c_a^b\big).
  \end{equation}
  By using \eqref{eq:1}, \eqref{eq:sim}, detailed balance
  \eqref{eq:db}, and \eqref{eq:sim} again  we obtain
  \begin{equation*}
    \begin{split}
     & \pi_N\Big[c_ac_b\Big(1+\frac{c_b^a}{c_b}\Big)
     \nabla_a f\nabla_b f\Big]\\
     &\qquad
     =-\frac{1}{2}\, \pi_N\big[ c_a\big(c_b+c_b^a\big)
     \nabla_a f\nabla_a\nabla_b f \big]
     = -\frac{1}{2}\, 
     \pi_N\big[ c_b \big(c_a+c_a^b\big) 
     \nabla_a f\nabla_a\nabla_b f \big]\\
     &\qquad
     =\frac{1}{2}\, \pi_N\big[c_b\big(c_a+c_a^b\big)
     \nabla_a f^b\nabla_a\nabla_b f \big]
     = \frac{1}{2}\,\pi_N\big[ c_a \big(c_b+c_b^a\big)
     \nabla_a f^b\nabla_a\nabla_b f \big].
    \end{split}
  \end{equation*}
  By averaging the previous equation and \eqref{eq:1} we get
  \begin{equation*}
    \begin{split}
     \pi_N\Big[c_ac_b\Big(1+\frac{c_b^a}{c_b}\Big)
     \nabla_a f\nabla_b f\Big] 
     &=
    \frac{1}{4}\, \pi_N\big[ c_a\big(c_b+c_b^a\big)
    \big(\nabla_a f^b-\nabla_a f\big) \nabla_a\nabla_b f \big]
    \\
    & =\frac{1}{4}\, \pi_N\big[c_a\big(c_b+c_b^a\big)
    (\nabla_a\nabla_b f)^2 \big]\geq 0
    \end{split}
  \end{equation*}
  which concludes the proof of the claim.

  \medskip 
  In order to prove \eqref{eq:gapp2} we observe that
  \begin{equation}
    \label{eq:2}
    \sum_{a\cap b\not=\emptyset}\pi_N\big[c_ac_b\nabla_a f\nabla_b f\big]=
    \sum_{a}\pi_N\big[c_a^2(\nabla_a
      f)^2\big]+\sum_{\substack{a\cap
        b\not=\emptyset\\a\not=b}}\pi_N\big[c_ac_b\nabla_a f\nabla_b
      f\big]. 
  \end{equation}
  Furthermore, given $a,b\in B_N$ such that $a\cap b\not=\emptyset$
  and $a\not=b$ there exists a unique \emph{triangle} $T$ such that
  $a,b\in T$.  A triangle here is an element of
  \begin{displaymath}
    \mathcal{T}_N:=\big\{\{a,b,c\}\subset B_N:|\{a,b,c\}|=3, a\cap
      b\not=\emptyset, a\cap c\not=\emptyset, b\cap
      c\not=\emptyset\big\}. 
  \end{displaymath}
 Therefore
  \begin{equation*}
    \sum_{\substack{a\cap
        b\not=\emptyset\\a\not=b}}\pi_N\big[c_ac_b\nabla_a f\nabla_b
      f\big]= 
    \sum_{T\in\mathcal{T}_N}
    \sum_{\substack{a,b\in T\\ a\not=b}}
    \pi_N\big[c_ac_b\nabla_a f\nabla_b f\big].
  \end{equation*}
  Note that
  \begin{equation*}
    \begin{split}
    &\sum_{T\in\mathcal{T}_N} \sum_{\substack{a,b\in T\\ a\not=b}}
    \pi_N\big[c_ac_b\nabla_a f\nabla_b f\big]
    \\
    &\qquad
    =\sum_{T\in\mathcal{T}_N}\sum_{a,b\in T}
    \pi_N\big[c_ac_b\nabla_a f\nabla_b f\big]
    -\sum_{T\in\mathcal{T}_N}\sum_{a\in T}
    \pi_N\big[c_a^2(\nabla_a f)^2\big]\\
    &\qquad=
    \sum_{T\in\mathcal{T}_N}\sum_{a,b\in T}
    \pi_N\big[c_ac_b\nabla_a f\nabla_b f\big]
    -\sum_{a} \big|\{T\in\mathcal{T}_N: T\ni a\}\big| \, 
    \pi_N\big[c_a^2(\nabla_a f)^2\big]\\
    &\qquad=\sum_{T\in\mathcal{T}_N}
    \sum_{a,b\in T}\pi_N \big[c_ac_b\nabla_a f\nabla_b f\big]
    -(N-2)\sum_{a}\pi_N\big[c_a^2(\nabla_a f)^2\big].
    \end{split}
  \end{equation*}
  By plugging this result in \eqref{eq:2} we get 
  \begin{displaymath}
    \sum_{a\cap b\not=\emptyset}\pi_N\big[c_ac_b\nabla_a f\nabla_b f\big]=
    \sum_{T\in\mathcal{T}_N}\sum_{a,b\in T}\pi_N\big[c_ac_b\nabla_a
    f\nabla_b f\big]-(N-3)\sum_{a}\pi_N\big[c_a^2(\nabla_a f)^2\big]. 
  \end{displaymath}
  For any $T\in\mathcal{T}_N$ define the set of vertexes of $T$ as
  $\tilde T:=\bigcup_{a\in T}a$. 
  Then
  \begin{equation}
    \label{tobepl}
    \sum_{a,b\in T}\pi_N\big[c_ac_b\nabla_a f\nabla_b f\big]=
    \pi_N\Big[\sum_{a,b\in T}\pi_N\big[c_ac_b\nabla_a f\nabla_b
    f \,\big|\,\sigma(z):z\not\in \tilde T\big]\Big]. 
  \end{equation}

  We prove in Lemma~\ref{lem:per} below that there exists a constant 
  $C_1(\beta)>0$ satisfying $\lim_{\beta\downarrow 0} C_1(\beta)=1$
  such that for any $N\geq 3$, any $f \colon S_N\to \mb R$, 
  and any $\sigma\in S_N$
  \begin{displaymath}
    \frac{N}{3}\sum_{a,b\in T}
    \pi_N\big[c_ac_b\nabla_a f\nabla_b f\,\big|\, 
    \sigma(z):z\not\in \tilde T\big]
    \geq 
    \frac{C_1(\beta)}{2} 
    \sum_{a\in T}\pi_N\big[c_a(\nabla_a f)^2 \, \big|\, 
    \sigma(z):z\not\in \tilde T\big]. 
  \end{displaymath}
  By plugging this bound into \eqref{tobepl}, we deduce 
  \begin{equation}
    \label{1.5}
    \begin{split}
    & \sum_{a\cap b\not=\emptyset}\pi_N\big[c_ac_b\nabla_a f\nabla_b
    f\big]
    \\
    &\qquad 
    \geq
    \frac{3C_1(\beta)}{2N}\sum_{T\in\mathcal{T}_N}\sum_{a\in
      T}\pi_N\big[c_a(\nabla_a
    f)^2\big]-(N-3)\sum_{a}\pi_N\big[c_a^2(\nabla_a f)^2\big]
    \\ 
    &\qquad=\frac{3C_1(\beta)(N-2)}{2N}\sum_{a}\pi_N\big[c_a(\nabla_a
    f)^2\big]-(N-3)\sum_{a}\pi_N\big[c_a^2(\nabla_a f)^2\big]
    \\ 
    &\qquad\geq\Big(\frac{3C_1(\beta)(N-2)}{2N}-(N-3)
    \sup_a\Vert c_a\Vert_\infty\Big)
    \sum_{a}\pi_N\big[c_a(\nabla_a f)^2\big].
    \end{split}
  \end{equation}
  In view of \eqref{ratN}, 
  \begin{equation}
    \label{brN}
    \sup_a\Vert c_a\Vert_\infty\leq
    \frac{1}{N}\exp\Big\{\frac{\beta}{2}\sup_{N,a}\Vert\nabla_a
      E_N\Vert_\infty\Big\}.
  \end{equation}
  Recalling the hypotheses $\sup_N\sup_{a}\Vert \nabla_a
  E_N\Vert_\infty<+\infty$, 
  from \eqref{1.5} we then deduce there exists a constant
  $C_2(\beta)>0$ satisfying $\lim_{\beta\downarrow 0} C_2(\beta)=1$ such
  that 
  \begin{equation}
    \label{eq:2primo}
    \sum_{a\cap b\not=\emptyset}\pi_N\big[c_ac_b\nabla_a f\nabla_b f\big]\geq
    \frac{C_2(\beta)}{2}\sum_{a}\pi_N\big[c_a(\nabla_a f)^2\big].
  \end{equation}

  To conclude the proof of \eqref{eq:gapp2} we show that the second
  term on its right hand side is, for $\beta$ small enough, of
  order $\beta$.  By Schwarz inequality and \eqref{eq:sim}
  \begin{equation*}
    \begin{split}
    & \Big| \frac{1}{2} 
    \sum_{\substack{ a,b\in B_N \\ a\cap b=\emptyset}}
    \pi_N\Big[c_ac_b\Big(1-\frac{c_b^a}{c_b}\Big)
    \nabla_a f\nabla_b f\Big]\Big|\leq
    \frac{1}{2}\sum_{ a\cap b=\emptyset}
    \pi_N\Big[c_ac_b \Big|1-\frac{c_b^a}{c_b}\Big|  
    \, |\nabla_a f| \, |\nabla_b f| \Big] 
    \\
    &\quad 
    \leq\frac{1}{4}
    \Big\{
    \sum_{ a\cap b=\emptyset} 
    \pi_N\Big[c_ac_b\Big|1-\frac{c_b^a}{c_b}\Big| 
    (\nabla_a f)^2 \Big]
    +\sum_{ a\cap b=\emptyset}
    \pi_N\Big[ c_a c_b \Big|1-\frac{c_b^a}{c_b}\Big| 
    (\nabla_b f)^2 \Big]\Big\}
    \\
    &\quad 
     =\frac{1}{2}\sum_{ a\cap b=\emptyset}
    \pi_N\Big[ c_a c_b\Big|1-\frac{c_b^a}{c_b}\Big| (\nabla_a f)^2 \Big]
    =\frac{1}{2}
    \sum_a\pi_N\Big[c_a(\nabla_a f)^2 
    \sum_{ b\,:\, b\cap a=\emptyset}c_b\Big|1-\frac{c_b^a}{c_b}\Big|\Big].
    \end{split}
  \end{equation*}

  The hypotheses of the theorem implies 
  $\sup_N\sup_{a,b}\Vert \nabla_a \nabla_b E_N\Vert_\infty<+\infty$. 
  Recalling \eqref{brN}, for $\beta$ small enough we then have
  \begin{equation*}
    \begin{split}
      \sum_{ b\,:\, b\cap a=\emptyset} c_b 
      \Big|1-\frac{c_b^a}{c_b}\Big|
      & \leq
      \sup_a\Vert c_a\Vert_\infty 
      \sum_{ b\,:\, b\cap a=\emptyset} 
      \big|1-e^{-\frac{\beta}{2}\nabla_a\nabla_b E_N}\big|
      \\
      &\leq\beta\sup_a\Vert c_a\Vert_\infty
      \sum_{ b\,:\, b\cap a=\emptyset}
      \big|\nabla_a\nabla_b E_N\big| 
      \leq C_3\beta
    \end{split}
  \end{equation*}
  for some constant $C_3$ independent of $N$. Therefore
  \begin{displaymath}
    \frac{1}{2}\sum_{\substack{ a,b\in B_N \\ a\cap b=\emptyset}} 
    \pi_N\Big[c_ac_b\Big(1-\frac{c_b^a}{c_b}\Big)\nabla_a f\nabla_b f\big]  
    \geq -\frac{C_3 \beta}{2}\sum_{a}\pi_N\big[c_a(\nabla_a f)^2\big],
  \end{displaymath}
  which together with \eqref{eq:2primo}  completes the proof
  of \eqref{eq:gapp2}.
\end{proof}

  \begin{lemma}
    \label{lem:per}
    Assume $\sup_N\sup_{a}\Vert \nabla_a E_N\Vert_\infty<+\infty$.
    Then there exists a constant $C_1(\beta)$ satisfying
    $\lim_{\beta\downarrow 0} C_1(\beta)=1$ such that for any
    $N\geq3$, any $T\in \mc T_N$, any $f\colon S_N\to \mb R$, and any
    $\sigma\in S_N$
    \begin{displaymath}
      \frac{N}{3}\sum_{a,b\in T}
      \pi_N\big[ c_ac_b\nabla_a f\nabla_b f 
        \, \big|\, \sigma(z) : z\not\in \tilde T\big]
        \geq 
        \frac{C_1(\beta)}{2}
        \sum_{a\in T}
        \pi_N\big[ c_a(\nabla_a f)^2 \, \big| \, 
        \sigma(z):z\not\in \tilde T\big]
    \end{displaymath}
    where we recall $\tilde T:=\bigcup_{a\in T}a$.
  \end{lemma}

  \begin{proof}
    The argument relies on two ingredients. The first is that, given a
    triangle $T\in\mc T_N$, the conditional probability 
    $\pi_N^\beta[\cdot \,|\, \sigma(z):\not\in \tilde T]$ is, for
    $\beta$ small enough, close to the uniform measure. 
    Namely,  there exist $C_4(\beta)$ satisfying
    $\lim_{\beta\downarrow 0} C_4(\beta)=1$ independent of $N$,
    $T\in\mc T_N$, and $\sigma\in S_N$, such that  
    \begin{equation}
      \label{eq:equ}
      \frac{1}{C_4(\beta)}\leq
      \frac{\pi_N^\beta\big[\sigma \,\big|\, 
        \sigma(z):z\not\in \tilde T \big]}
      {\pi_N^0\big[\sigma\,\big|\,\sigma(z):z\not\in \tilde T\big]} 
      \leq C_4(\beta).
    \end{equation}  
    The second ingredient is that the spectral gap of the interchange
    process on a graph with $3$ vertices is equal to $1$. 
    This statement readily implies 
    \begin{equation}
      \label{eq:gaps3}
      \pi_N^0\big[ f,f \,\big|\, \sigma(z): z\not\in \tilde T] \big]
      \leq \frac{1}{6} 
      \sum_{a\in T}\pi_N^0\big[(\nabla_af)^2 \,\big| \,\sigma(z):z
      \not\in \tilde T\big]. 
    \end{equation}

    We first show that \eqref{eq:equ} and \eqref{eq:gaps3} imply the thesis. 
    Since $\pi_N^0\big(\cdot \,\big|\, \sigma(z):z\not\in \tilde T\big)$ is the
    uniform measure on a set of cardinality $6$, in view of
    \eqref{eq:gaps3}
    \begin{equation*}
    \begin{split}
     &\pi_N^\beta\big[f,f\,\big|\, \sigma(z):z\not\in \tilde T\big]
     \leq
    C_4(\beta)^2 
    \pi_N^0\big[f,f \,\big|\, \sigma(z):z\not\in \tilde T\big]
    \\
    &\qquad \leq
    \frac{C_4(\beta)^2}{6}
    \sum_{a\in T}
    \pi_N^0\big[(\nabla_af)^2\,\big|\,\sigma(z):z\not\in \tilde
    T\big]
    \\ 
    &\qquad \leq\frac 13 C_4(\beta)^3 N
    e^{\frac{\beta}{2}\sup_{N,a} \Vert \nabla_a E_N \Vert_\infty}
    \, \frac 12 \sum_{a\in T}\pi_N^\beta\big[c_a(\nabla_af)^2\,\big|\,
    \sigma(z):z\not\in \tilde T\big].
    \end{split}
  \end{equation*}
  Whence, by using the characterization of the spectral gap given in
  \eqref{eq:gap2}, 
  \begin{equation*}
    \begin{split}
    & \sum_{a,b\in T} 
    \pi_N^\beta\big[c_ac_b\nabla_af\nabla_bf \,\big|\, 
    \sigma(z):z\not\in \tilde T\big]\\
    & \qquad \geq \frac{3}{C_4(\beta)^3 \, N} 
      \, \exp\big\{- \tfrac{\beta}{2} \sup_{N,a} \Vert \nabla_a E_N
      \Vert_\infty\big\}
    \, \frac 12 \sum_{a\in T}
    \pi_N^\beta\big[c_a(\nabla_af)^2\,\big|\,\sigma(z):z\not\in \tilde T\big]
    \end{split}
  \end{equation*}
  which, for a suitable $C_1(\beta)$, is the thesis of the lemma. 

  The estimate \eqref{eq:equ} follows from standard arguments. Firstly
  note that for any $a\in T$
  \begin{displaymath}
    \frac{\pi_N^\beta \big[\sigma^a \,\big|\, 
      \sigma(z):z\not\in \tilde T\big]}
    {\pi_N^\beta \big[\sigma\,\big|\, \sigma(z):z\not\in \tilde T\big]}= 
    \exp\big\{-\beta\, \nabla_a E_N(\sigma) \big\}.
  \end{displaymath}
  Therefore, by observing that any two given permutations in $S_3$ can
  be connected at most by two transpositions, a telescopic argument yields
  that for any $\sigma,\sigma'\in S_N$ such that
  $\sigma(z)=\sigma'(z)$ for $z\not\in\tilde T$
  \begin{displaymath}
    \exp\big\{-2\beta \, \sup_{N,a} \Vert\nabla_a
    E_N\Vert_\infty\big\}
    \leq 
    \frac{\pi_N^\beta\big[\sigma^\prime\,\big|\, \sigma^\prime(z):z\not\in \tilde
      T\big]}
    {\pi_N^\beta \big[\sigma \,\big|\, \sigma(z):z\not\in \tilde
      T \big]}
    \leq 
     \exp\big\{2\beta\, \sup_{N,a}\Vert\nabla_a E_N\Vert_\infty\big\}.
  \end{displaymath}
  By averaging the above inequality over $\sigma'$ the bound \eqref{eq:equ}
  follows.

  The spectral gap of the interchange process on $\{1,2,3\}$ can be
  deduced from the general results in \cite{CLR,DS}. An elementary
  proof can however also be obtained by writing out the $6\times 6$  
  matrix corresponding to the generator and computing its
  eigenvalues, as in the example at page 50 of \cite{Di}. 
  We order the $6$ permutations of $S_3$ as
  $\binom{123}{123}$, $\binom{123}{132}$, $\binom{123}{231}$, 
  $\binom{123}{213}$, $\binom{123}{312}$, and $\binom{123}{321}$. 
  With this choice, the generator of the interchange process is
  represented by the matrix  
  \begin{equation*}
    \begin{pmatrix}
      -1 & \phantom{-}\frac 13 & \phantom{-} 0 & \phantom{-}\frac 13
      &\phantom{-} 0 & 
      \phantom{-}\frac 13 \\ 
      \phantom{-}\frac 13 & -1 & \phantom{-}\frac 13 & \phantom{-} 0 &
      \phantom{-}\frac 13  &\phantom{-} 0  \\ 
      \phantom{-} 0 & \phantom{-}\frac 13 & -1 & \phantom{-}\frac 13 &
      \phantom{-} 0 & 
      \phantom{-}\frac 13  \\ 
      \phantom{-}\frac 13 & \phantom{-} 0 & \phantom{-}\frac 13 & -1 &
      \phantom{-}\frac 13 & \phantom{-} 0 \\ 
      \phantom{-} 0& \phantom{-}\frac 13 & \phantom{-} 0 &
      \phantom{-}\frac 13 & -1 & 
      \phantom{-}\frac 13  \\ 
      \phantom{-}\frac 13 & \phantom{-} 0& \phantom{-}\frac 13 &
      \phantom{-} 0 & 
      \phantom{-}\frac 13 & -1 \\ 
    \end{pmatrix}
  \end{equation*}
  whose eigenvalues are $0$ (simple), $-1$ (with
  multiplicity four), and  $-2$ (simple).
\end{proof}

We finally show, as a corollary of the previous result, that the
spectral gap of the ABC dynamics on the complete graph is of order one.

\begin{proof}[Proof of Lemma~\ref{t:cor}]
  Fix $N$ multiple of three and let $\chi_N\colon
  S_N\to\Omega_N$ be the projection defined by
  \begin{equation*}
    (\chi_N \sigma)  \, (x) :=
    \begin{cases}
      A & \textrm{ if } \sigma(x)\equiv 1 \mod 3\\
      B & \textrm{ if } \sigma(x)\equiv 2 \mod 3\\
      C & \textrm{ if } \sigma(x)\equiv 3 \mod 3.
    \end{cases}
  \end{equation*}
  Namely, $\chi_N$ resolves only three out of the original $N$
  colors. 
  Recalling \eqref{ham}, let $E_N\colon S_N\to\mb R$ be defined
  by $E_N := N\, H_N \circ \chi_N$. For this choice 
  the ABC dynamics with long jumps , i.e.\ the process
  with the generator \eqref{Lc=}, can be realized as the
  $\chi_N$-projection of the process with generator \eqref{genG}.
  In particular, $\nu_N^\beta = \pi_N^\beta\circ \chi_N^{-1}$ and 
  $\gap\big( \ms L_N^\beta  \big)  \ge \gap\big(\mc G_N^\beta\big)$.
  Since $H_N \circ \chi_N$ satisfies the hypotheses in
  Theorem~\ref{t:a1}, the statement follows. 
\end{proof}

\subsection*{Acknowledgements.}
We express our warmest thanks to C.~Furtlehner who poin\-ted
out to us the possible occurrence of a metastable behavior for the ABC
model. We also thank T.~Bodineau and J.L.~Lebowitz for useful discussions.  
N.C.\ and G.P.\ acknowledges the financial support of the European Research
Council through the ``Advanced Grant'' PTRELSS 228032.

\end{document}